\documentclass{amsart}

\newtheorem{thm}{Theorem}
\newtheorem{prop}[thm]{Proposition}
\newtheorem{lem}[thm]{Lemma}

\usepackage{wasysym}
\usepackage{amsmath}
\usepackage{amssymb}
\usepackage{amsthm}
\usepackage{tikz}
\usetikzlibrary{matrix,arrows,backgrounds,shapes.misc,shapes.geometric,patterns,calc,positioning}
\usetikzlibrary{calc,shapes}
\usepackage{wrapfig}
\usepackage{epsfig}  
\usepackage{color}	 %
\input{xy}
\xyoption{poly}
\xyoption{2cell}
\xyoption{all}

\usepackage{lscape}

\usetikzlibrary{calc,shapes}
\usetikzlibrary{matrix,arrows,backgrounds,shapes.misc,shapes.geometric,patterns,calc,positioning}
\usetikzlibrary{calc,shapes}

\newcommand{\za}{\alpha}

\newcommand{\zd}{\delta}

\newcommand{\ze}{\epsilon}
\newcommand{\zg}{\gamma}

\newcommand{\zs}{\sigma}
\newcommand{\zS}{\Sigma}

\newcommand{\cala}{\mathcal{A}}
\newcommand{\calu}{\mathcal{U}}

\newcommand{\calg}{\mathcal{G}}

\newcommand{\bx}{{\bf x}}

\def\B{\mathcal{B}}

\makeatletter
\def\Ddots{\mathinner{\mkern1mu\raise\p@
\vbox{\kern7\p@\hbox{.}}\mkern2mu
\raise4\p@\hbox{.}\mkern2mu\raise7\p@\hbox{.}\mkern1mu}}
\makeatother

\begin{document}

\title{On cluster algebras from unpunctured surfaces with one marked point}
\author{Ilke Canakci}
\address{Department of Mathematics, 
University of Leicester,
University Road, 
Leicester LE1 7RH, 
United Kingdom}
\email{ic74@le.ac.uk}
\author{Kyungyong Lee}
\address{Department of Mathematics, Wayne State University, Detroit, MI 48202, USA\\
{ -- Center for Mathematical Challenges, Korea Institute for Advanced Study, 85 Hoegiro, Dongdaemun-gu
Seoul 130-722
Republic of Korea}}
\email{klee@math.wayne.edu \\
{klee1@kias.re.kr}}
\author{Ralf Schiffler}\thanks{The first author was {supported by by EPSRC grant number EP/K026364/1, UK} and the University of Leicester, {the second author was supported by  Wayne State University and the Korea Institute for Advanced Study,} and the last author was supported by the NSF grants DMS-1254567, DMS-1101377 and by the University of Connecticut.}
\address{Department of Mathematics, University of Connecticut, 
Storrs, CT 06269-3009, USA}
\email{schiffler@math.uconn.edu}

\maketitle
\begin{abstract}{We extend the construction of canonical bases for cluster algebras from unpunctured surfaces to the case where the number of marked points  on the boundary is one, and we show that the cluster algebra is equal to the upper cluster algebra in this case.}
\end{abstract}
\section{Introduction}
Cluster algebras were introduced in \cite{FZ1}, and further developed in \cite{FZ2,BFZ,FZ4}, motivated by combinatorial aspects of canonical bases in Lie theory \cite{L1,L2}. A cluster algebra is a subalgebra of a field of rational functions in several variables, and it is given by constructing a distinguished set of generators, the \emph{cluster variables}. These cluster variables are constructed recursively and their computation is rather complicated in general. By construction, the cluster variables are rational functions, but Fomin and Zelevinsky showed in \cite{FZ1} that they are Laurent polynomials with integer coefficients. Moreover, these coefficients are known to be non-negative \cite{LS4}.

An important class of cluster algebras is given by cluster algebras of surface type \cite{GSV,FG1,FG2,FST,FT}. From a classification point of view, this class is very important, since it has been shown in \cite{FeShTu} that almost all (skew-symmetric) mutation finite cluster algebras are of surface type. For generalizations to the skew-symmetrizable case see \cite{FeShTu2,FeShTu3}. The closely related surface skein algebras were studied in \cite{M,T}.

If $\cala$ is a cluster algebra of surface type, then there exists a surface with (possibly empty) boundary and marked points such that the cluster variables of $\cala$ are in bijection with certain isotopy classes of curves, called \emph{arcs}, in the surface. {Marked points in the interior of the surface are called \emph{punctures}, and the surface is \emph{unpunctured} if all  marked points lie on the boundary}. Moreover, the relations between the cluster variables are given by the crossing patterns of the arcs in the surface. In \cite{MSW}, building on earlier work \cite{S2,ST,S3,MS}, the authors gave a combinatorial formula for the cluster variables in cluster algebras of surface type. In the sequel \cite{MSW2}, the formula was the key ingredient for the construction of two bases for the cluster algebra, in the case where the surface has no punctures and has at least 2 marked points.

\subsection{Bases} Our main result, Theorem \ref{thm 1}, shows that the basis construction of \cite{MSW2} also applies to surfaces {without punctures} and with exactly one marked point. 
The proof of this result consists in showing that the Laurent polynomials associated to the essential loops in the surface  are elements of the cluster algebra. This is shown by exhibiting certain identities in the cluster algebra, that allow to write the Laurent polynomials in question as polynomials in cluster variables. 
The main ingredients for the proof of these identities are the snake graph calculus developed in {{\cite{CS,CS2,CS3}}}, and
the skein relations proved in \cite{MW}.

\subsection{Upper cluster algebras} As an application, we study the relationship between cluster algebras and upper cluster algebras. 
To define a cluster algebra $\cala$, one needs to specify an initial seed $\zS=(\mathbf{x},B)$ consisting of a cluster $\mathbf{x}$ and an exchange matrix $B$. In \cite{BFZ}, the authors introduced the  concept of the upper cluster algebra $\calu$ associated to the seed $\zS$. 
 Recall that  ${\calu}$ consists of all elements of the ambient field {$\mathbb{Q}(x_1,\ldots,x_n)$} which are Laurent polynomials over $\mathbb{Z}$ in the cluster variables from any seed in $\cala$.
It follows directly from the Laurent phenomenon that the upper cluster algebra contains the cluster algebra as a subalgebra. However, the question whether the cluster algebra is equal to the upper cluster algebra is subtle and does not have a uniform answer. Already in \cite{BFZ} it is shown that $\cala=\calu$ for all acyclic types and, on the other hand, $\cala\ne\calu$ for the rank 3 case given by the once-punctured torus.
Muller introduced the notion of locally acyclic cluster algebras in \cite{M2} and showed that $\cala=\calu$ for all locally acyclic types {in \cite{M3}}. It is shown in \cite{MSp}, that  cluster algebras  of Grassmannians are locally acyclic, hence $\cala=\calu$ in this case. {Goodearl and Yakimov announced that $\cala=\calu$ for double Bruhat cells \cite{GY}.}

For cluster algebras of finite mutation type, the following results were known. 
\bigskip
\begin{enumerate}
\item 
{ $\cala=\calu$ in the following types: }

\begin{itemize}
\item[-] surfaces with non-empty boundary and at least two marked points. This has been shown in \cite{MSW2} for  unpunctured surfaces, and in \cite{M2} if  at least two marked points are on the boundary. The case where only one point is on the boundary can be reduced to the case with two marked points on the boundary {using the Louise property of \cite{MSp}}, see Proposition \ref{prop 4}.
\item[-]  the exceptional types $\mathbb{E}_6,\mathbb{E}_7,\mathbb{E}_8,\widetilde{\mathbb{E}}_6,\widetilde{\mathbb{E}}_7,\widetilde{\mathbb{E}}_8$,  by \cite{BFZ}, since these types are acyclic.
\item[-]  the exceptional types  $\widetilde{\widetilde{\mathbb{E}}}_6,\widetilde{\widetilde{\mathbb{E}}}_7,\widetilde{\widetilde{\mathbb{E}}}_8$ and $\mathbb{X}_6$, {by \cite{M}, since they are locally acyclic}, or by \cite{MSp}, since they satisfy the Louise property.
\end{itemize}
 
\bigskip
\item{
$\cala\ne \calu$ in the following types:}
\begin{itemize}
\item[-]  surfaces without boundary and with exactly one puncture. This has been shown for the torus in \cite{BFZ} and for higher genus in \cite{Lad}.
\item[-]  the exceptional type $\mathbb{X}_7$,  \cite{Mprivate}.
\end{itemize}
\end{enumerate}
\bigskip

As an application of Theorem \ref{thm 1}, we prove in Theorem \ref{thm 2} that 

\centerline{$\cala=\calu$ for surfaces without punctures and exactly one marked point.}

This leaves the question open only for  surfaces without boundary and at least two punctures.

Summarizing, we have the following result.

\begin{thm}
 For all unpunctured surfaces, the cluster algebra is equal to the upper cluster algebra.
\end{thm}

\subsection{Maximal green sequences}
{Let $Q$ be a quiver and let $\bar Q$ be the quiver obtained from $Q$ by adding a vertex $i'$ and an arrow $i\to i'$, for each vertex $i$ of $Q$. The original vertices of $Q$ are called \emph{unfrozen} (or \emph{mutable}) and the new vertices are called \emph{frozen}. A \emph{green sequence} is a sequence of mutations starting at $\bar Q$ such that at each step the mutation is performed at an unfrozen vertex $i$, for which there is no arrow $j'\to i$, from a frozen vertex $j'$. A green sequence is \emph{maximal} if it produces a quiver which does not contain such a vertex $i$.

Maximal green sequences have been introduced in \cite{K}. They are related to quantum dilogarithm identities, DT invariants and BPS states, see \cite{BDP} and the references therein. 

The question whether or not {there exists a quiver in the mutation class of $Q$ which} admits a maximal green sequence seems related to the question whether or not $\cala(Q)=\calu(Q)$. Indeed, 
{given a surface $(S,M)$ which is either
a surface with boundary, or a sphere with at least 4 punctures or a torus with at least 2 punctures, there exists a triangulation $T$ whose quiver $Q_T$ admits a
maximal green sequence}. This has been shown in \cite{ACCERV}, see also \cite{BDP} for the sphere.
Moreover, {for all the exceptional types except for $\mathbb{X}_7$ there exists a quiver that admits a maximal green sequence \cite{ACCERV}, whereas the mutation class  $\mathbb{X}_7$ does not contain a quiver that admits a maximal green sequence \cite{Se}.}

On the other hand, it is shown in \cite[Proposition 8.1]{BDP} that if $Q$ admits a non-degenerate potential such that the Jacobian algebra is infinite dimensional, then $Q$ does not admit maximal green sequences. Combining this result with \cite[Proposition 9.13]{GLS}, we see that quivers from closed surfaces with exactly one puncture do not admit maximal green sequences. {The question seems to be open for  closed surfaces of genus at least 2 and with at least two punctures.} 

{Beyond finite mutation type, explicit maximal green sequences are announced by Yakimov for double Bruhat cells.}

In all the known cases, maximal green sequences exist if and only if the cluster algebra is equal to the upper cluster algebra. Our Theorem \ref{thm 2} confirms this observation for unpunctured surfaces with exactly one marked point.}

Acknowledgements: Part of this work has been carried out at the Centre de Recherche Math\'ematiques in Montr\'eal, and we thank the CRM for their hospitality. We also thank Greg Muller for stimulating discussions. 

\section{Main Result} Let $\mathcal{A}={\cala(S,M)}$ be a cluster algebra of an unpunctured surface ${S}$ with marked points ${M}$ whose coefficient system is such that the initial exchange matrix has maximal rank. The main result of \cite{MSW2} is the construction of two bases for $\cala$ under the assumption that the number of marked points is at least 2. These are the \emph{bangles basis} $\B_0$ and the \emph{bracelet basis} $\B$.
The assumption that the number of marked points is at least 2 was used in \cite{MSW2}  only to show that all elements of $\B_0$ and $\B$ are actually inside the cluster algebra, and not only in the upper cluster algebra. The proof of the spanning property and of  linear independence does not rely on the number of marked points.

The following theorem removes the assumption on the number of marked points.

\begin{thm} \label{thm 1} Let $\mathcal{A}$ be the cluster algebra of an unpunctured surface with exactly one marked point and with arbitrary coefficients. Then both $\B$ and $\B_0$ are bases of the cluster algebra $\mathcal{A}$.
\end{thm}

{For the proof of this theorem, we need the following two Lemmas. An \emph{essential loop} in the surface $(S,M)$ is a closed curve in $S$ which is disjoint from the boundary of $S$, which is not contractible and does not have any self-crossings.  Since $|M|=1$, the surface $S$ has exactly one boundary component. In \cite{MSW2}, the authors associate to every essential loop a Laurent polynomial  given as a sum over perfect matchings of the corresponding band graph.}
\begin{lem}\label{lem loop}
  The Laurent polynomial $L$ of the essential loop around the boundary  is in the cluster algebra $\cala$.
\end{lem}

\begin{lem}\label{lem loop2}
 The Laurent polynomial $x_\zeta$ of every essential loop $\zeta$ is in the cluster algebra $\cala$.
\end{lem}

Of course the second lemma implies the first, but the first is needed in the proof of the second. We prefer stating them separately, because their proofs use different techniques. 
 The proof  of Lemma \ref{lem loop} is given in {Section} \ref{sect proof} and the proof of Lemma \ref{lem loop2} in {Section} \ref{sect proof2}.

\begin{proof}[Proof of Theorem \ref{thm 1}]

 Let $T$ be a triangulation of the surface, and let $B_T$ be the associated exchange matrix. Since there are no punctures and the unique boundary component has an odd number of marked points (namely one), it follows from \cite[Theorem 14.3]{FST} that the rank of the matrix $B_T$ is maximal for every coefficient system. Therefore the results of \cite{MSW2} show that both $\B_0$ and $\B$  are linearly independent and that every element of $\cala$ is a linear combination of elements of $\B_0$ and also a linear combination of elements of $\B$. 
 
 It remains to show that $\B_0$ and $\B$ are subsets of $\cala$. 
 
 Lemma \ref{lem loop2} implies that every element of $\B_0$ is in $\cala$. 
 Moreover the Laurent polynomials associated to the bracelets in $\B$ can be written as Chebyshev polynomials in the Laurent polynomials associated to the essential loops in $\B_0$, see \cite[Proposition 4.2]{MSW2}. Thus $\B$ is also a subset of $\cala$. 
\end{proof}

The following result follows directly from Theorem \ref{thm 1}.
\begin{thm}\label{thm 2}
 Let $\mathcal{A}$ be the cluster algebra of an unpunctured surface with exactly one marked point and with trivial coefficients, and let $\calu$ be its upper cluster algebra. Then
 \[\cala=\calu.\]
\end{thm}
 
\begin{proof}
 It follows from  \cite[Theorem 12.3(ii)]{FG1}
  that $\B_0$ is a basis of $\calu$. Theorem~\ref{thm 1} implies that $\B_0\subset \cala$, and thus $\calu\subset\cala$. The other inclusion always holds.
\end{proof}

\section{Proof of Lemma \ref{lem loop}}\label{sect proof}
We {{prove}} the lemma first for genus 1, then genus 2 and then for higher genus.
Throughout the section, we denote by $\cala_\bullet$ the cluster algebra with principal coefficients in the initial seed corresponding to our triangulation $T$ and by $\hat\cala$ a cluster algebra with arbitrary coefficient system and initial seed corresponding to $T$. More precisely, let ${\mathbb{P}}=\textup{Trop}(y_1,\ldots,y_n)$ be the tropical semifield and ${\mathbf{y}}=(y_1,\ldots, y_n)$ be the initial coefficient tuple consisting of the generators of $\mathbb{P}$.
Then  $\cala_\bullet$ is the cluster algebra with initial seed $(\mathbf{x}_T,{\mathbf{y}},B_T)$.
On the other hand, let $\hat{\mathbb{P}}$ be any semifield and $\hat{\mathbf{y}}=(\hat y_1,\ldots,\hat y_n)$ be any coefficient tuple, thus $\hat y_i\in \hat{\mathbb{P}}$. Then $\hat\cala$ is the cluster algebra with initial seed $(\mathbf{x}_T,\hat{\mathbf{y}},B_T)$. 

{Recall that every arc or {loop} in  $(S,M)$ corresponds to a snake or band graph and to a Laurent polynomial given by perfect matchings of the snake or band graph. A band graph is obtained from a snake graph by identifying an edge in the first tile with an edge in the last tile. In the figures, this identification is represented by marking the vertices of the identified edges with bullets. For precise definitions, we refer to \cite{MSW2}. For the snake graph calculus used in this section, see \cite{CS2}.}

\subsection{Genus 1}
Let $(S,M)$ be  a surface of genus 1 with one boundary component and one marked point. {This case is known in the literature as `the dreaded torus', because, among cluster algebras of surface type, it was the smallest example where the question if the cluster algebra is equal to the upper cluster algebra was open, see \cite[Remark 7.2.3]{MM}.}  Fix the triangulation  $T$ shown on the left in Figure \ref{triangulationgenus2}, and let $\mathbf{x}_T=(x_1,\ldots,x_{4})$ be the corresponding cluster in $\cala_\bullet$. 

\begin{figure}[ht]
\begin{center}
\scalebox{0.4} { \Huge 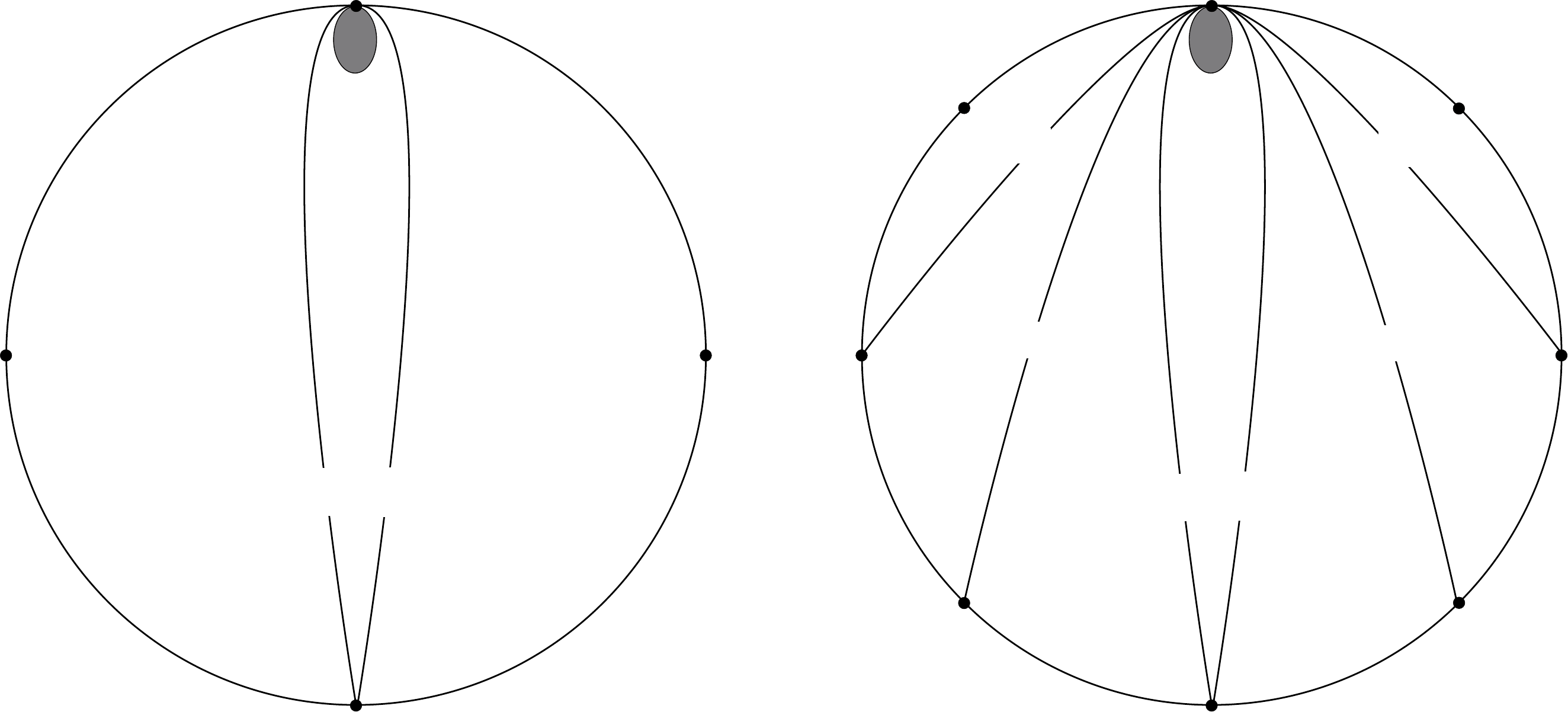}
 \caption{A triangulation of $(S,M)$ in genus 1 (left) and a triangulation in genus 2 (right)}
 \label{triangulationgenus2}
 \end{center}
\end{figure}

%

\begin{figure}[ht]
\begin{center}
\scalebox{0.57} { \large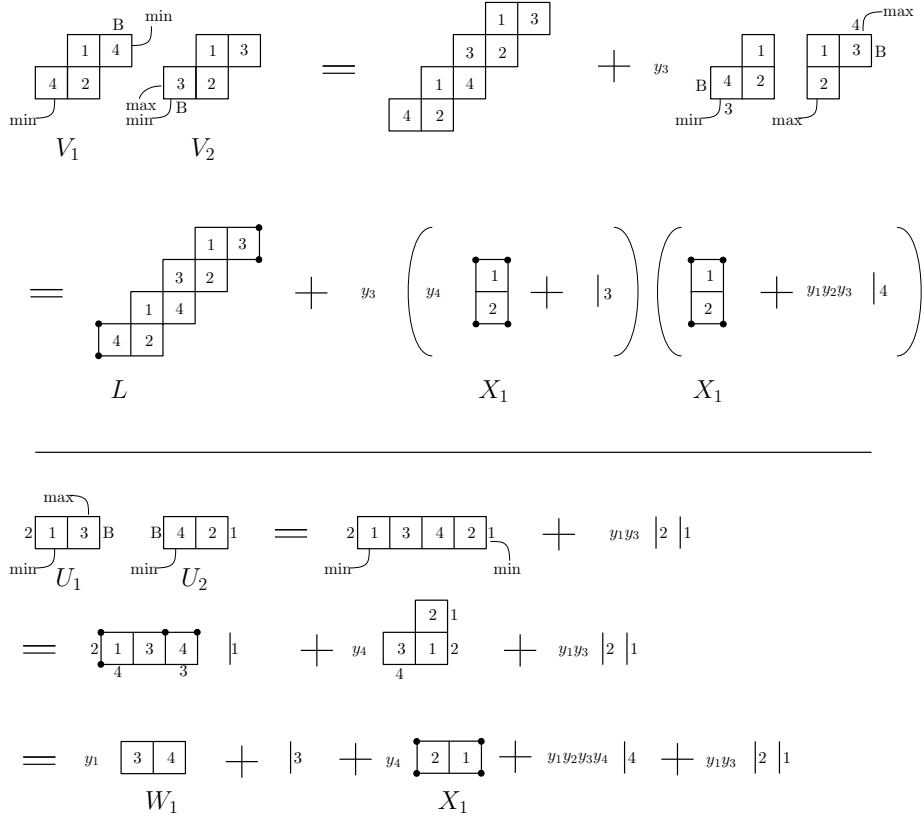}
 \caption {Snakegraph calculus with principal coefficients. The edge labels `min', `max' refer to minimal and maximal matchings, which is needed only to determine the $y$-coefficients. {The edge label $B$ refers to the boundary.}}
 \label{snakegraphcalculusgenus1acoeff}
 \end{center}
\end{figure}


Let $L$ be the Laurent polynomial associated to the loop around the boundary. Its band graph is the first graph in the second row  in Figure \ref{snakegraphcalculusgenus1acoeff}.
 
\begin{lem} 
$ L\in \cala_\bullet$.
\end{lem}
\begin{proof}
 Let $V_1$ be the cluster variable given by the arc that starts at the marked point in direction between the arc 3 and the boundary and then crosses the arcs 4,2,1,4 in order and ends at the marked point as it started from the direction between the arc 3 and the boundary.  
Let $V_2$ be the cluster variable given by the symmetric arc, that is, the arc that starts at the marked point in direction between the arc 4 and the boundary and then crosses the arcs 3,1,2,3 in order and ends at the marked point as it started from the direction between the arc 4 and the boundary.  The snake graphs of these two cluster variables are the first two snake graphs in the top row of 
 Figure \ref{snakegraphcalculusgenus1acoeff}.
 The first two equations in that figure  use snake graph calculus of \cite{CS} and \cite{CS2} to compute the product  $V_1 V_2$ showing that
  \begin{equation}\label{eq1}
 V_1\, V_2 = L  + y_3(y_4\,X_1  +x_{3})(X_1+y_1y_2y_3\,x_{4}),
\end{equation}
where  $X_1$ is the Laurent polynomial defined by the corresponding band  graph in the figure.
The first equation is the resolution of  the grafting operation \cite[section 2.5 case 2]{CS} and the second equation is the resolution of three self-grafting operations \cite[section 3.4]{CS2}. The graphs of $L$ and $X_1$ are band graphs.
 
Thus in order to show that $L\in \cala_\bullet$ it suffices to show that $X_1\in \cala_\bullet$. This is done in the second computation in Figure \ref{snakegraphcalculusgenus1acoeff}, where we compute the product $U_1U_2$ of the two cluster variables  given by the arcs that cross 1,3 and 4,2 respectively, showing that
\begin{equation} \label{eq2}  U_1U_2 = y_1\, W_1+x_3+  y_4\,X_1+y_1y_2y_3y_4\,x_4+y_1y_3\,x_1x_2,
\end{equation}
where  $W_1$ is the Laurent polynomial defined by the corresponding snake  graph in the figure.
{{One can show that $U_1, U_2$ are cluster variables by checking that  the associated arcs in the surface do not have a self-crossing, or by using \cite{CS2} to show that their associated snake graphs do not have any self-crossing overlap.}} The first equation in this computation is the resolution of a grafting, the second equation is a self-grafting {{\cite[section 3.3]{CS2}}} and the third is a grafting {{of a band graph with a single edge \cite{CS3}}}  producing $y_1W_1+x_3$, as well as a self-grafting to obtain $y_4X_1 +y_1y_2y_3y_4x_4$. Note that $W_1$ is a cluster variable corresponding to the arc crossing 3 and 4. This shows that 
{$y_4 X_1\in \cala_\bullet$. Since $y_4$ is an element of the group $  \mathbb{P}$, its inverse $y_4^{-1}$ also is in $\mathbb{P}$, and therefore   $X_1\in \cala_\bullet$, and thus  $L\in\cala_\bullet$. }
 \end{proof}

%

 
\begin{lem}
 $L\in\hat\cala$.
\end{lem}
\begin{proof}
 
%
Using Fomin-Zelevinsky's separation of addition formula \cite[Theorem 3.7]{FZ4} we see that the elements of ${\hat\cala}$ can be computed from the elements in ${\cala}$ by replacing the principal coefficients $y_i$ by the corresponding $\hat y_i$ and by dividing by the $F$-polynomial evaluated over $\hat{\mathbb{P}}$ in $\hat y_i$.
For example, 
\[\hat U_1 =\frac{ U_1(x_1,\ldots,x_n;\hat y_1,\ldots,\hat y_n) }{ \hat F_{U_1}},\]
where $\hat F_{U_1}=F_{U_1}|_{\hat{\mathbb{P}}}(\hat y_1,\ldots,\hat y_n)\in\hat{\mathbb{P}}\subset\hat\cala.$

Evaluating the expressions in equations (\ref{eq1}) and (\ref{eq2}) in $\hat{\mathbf{y}}$ clearly preserves the identities. Then multiplying with the $F$-polynomials we get

{\begin{equation}\label{eq1hat}
 \hat V_1\,\hat V_2 \hat F_{V_1}\hat F_{V_2}= \hat L\hat F_{L}  + \hat y_3(\hat y_4\,\hat X_1\hat F_{X_1}  +x_{3})(\hat X_1\hat F_{X_1} +\hat y_1\hat y_2\hat y_3\,x_{4}),
\end{equation}
\begin{equation} \label{eq2hat}
 \hat  U_1\hat U_2{{\hat F_{U_1}}}\hat F_{U_2} = \hat y_1\, \hat W_1\hat F_{W_1}+x_3+\hat y_4\,\hat X_1\hat F_{X_1} +\hat y_1\hat y_2\hat y_3\hat y_4\,x_4+\hat y_1\hat y_3\,x_1x_2,
\end{equation}
}

From equation (\ref{eq2hat}) we see that {$\hat y_4\,\hat X_1\hat F_{X_1} \in\hat\cala$, and since $\hat y_4 \in\hat{\mathbb{P}}$, and thus $ \hat y_4^{-1} \in \hat{\mathbb{P}}$, this also implies that $\hat X_1\hat F_{X_1} \in\hat\cala$.} Now equation (\ref{eq1hat}) implies that $\hat L\hat F_{L}\in \hat\cala$, and, since $\hat F_{L}\in \hat{\mathbb{P}}$, this shows that $\hat L\in \hat \cala.$
\end{proof}\subsection{Genus 2}
Let $(S,M)$ be  a surface of genus 2 with one boundary component and one marked point. Fix the triangulation  {{$T$}} shown on the right in Figure \ref{triangulationgenus2}, and let $\mathbf{x}_T=(x_1,\ldots,x_{10})$ be the corresponding cluster in $\mathcal{A}$. 

Let $L$ be the Laurent polynomial associated to the loop around the boundary. Its band graph is the first graph in the last row  in Figure \ref{snakegraphcalculusgenus2acoeff}. 

\begin{lem} \label{lem 8}
$ L\in\cala_\bullet$.
\end{lem}
\begin{proof}
{Let $V_1$ be the cluster variable obtained from $\bx$ by {the mutation sequence 8,9,10,2,1,9,4,6,3, and let $V_2$ be the cluster variable obtained from $\bx$ by the mutation sequence 7,6,5,1,2,6,3,9,4.}} Then $V_1$ is the cluster variable corresponding to the arc that crosses $8,9,10,2,1,10,4,6,3,8$ and $V_2$ is the cluster variable corresponding to the arc that crosses $7,4,9,3,5,2,1,5,6,7$.
The corresponding snake graphs
 $\calg_1$ and $\calg_2 $ are illustrated on the left hand side of the equation  in Figure \ref{snakegraphcalculusgenus2acoeff}. 
 
{{Let}} $X_i$ be the Laurent polynomial defined  by the band graph obtained from $\calg_i$, {{for $i=1,2$,}} by deleting the first and the last tile and then glueing. These band graphs are illustrated in the last row of Figure  \ref{snakegraphcalculusgenus2acoeff}.

A simple computation using snake graph calculus shows that 
\[V_1\, V_2 = L  + y_7(y_8X_1  +x_{7})(X_2 +y_1y_2y_3y_4y_5^2y_6y_7y_9\,x_{8}).\]
This computation is given in Figure \ref{snakegraphcalculusgenus2acoeff}. 
The first equation is the resolution of a grafting, and the second is the resolution of three self-graftings.
 
In order to show that $L\in \cala_\bullet$, it suffices to show that $X_1$ and $X_2$ are in the cluster algebra and, by symmetry, it is enough to show that $X_1$ is.


Let $U_1$ and $U_2$ be the Laurent polynomials associated to the snake graphs on the left hand side of the equation in Figure \ref{snakegraphcalculusgenus2ccoeff}. It follows from \cite{CS2} that none of these two snake graphs  has a self-crossing, which implies that $U_1$ and $U_2$ are cluster variables. Alternatively, we can see that $U_1$ and $U_2$ are cluster variables since each corresponds to an arc in the surface, see Figure~\ref{Arcsingenus2}. {{Indeed, the variable}} $U_1$ corresponds to the arc starting from the marked point into the triangle
with sides 3,9,8,
 then crossing in order the arcs 3,6,4,10,1,5,6,7 and ending at the marked point coming from  the  triangle with sides $B,8,7$, where $B$ denotes the boundary segment; and the variable
$U_2$  corresponds to the arc starting from the marked point into the triangle with sides $B,8,7$, then crossing in order the arcs 8,9,10,2 and ending at the marked point coming from the triangle with sides 1,2,5. 
Since these two curves do not have self-crossings, they are indeed arcs, and hence $U_1$ and $U_2$ are cluster variables. 

Another simple calculation using snake graph calculus shows that 
\[U_1U_2 =  y_1\,W_1+x_7+y_8\,X_1+y_1y_8\,W_2+ y_1y_5y_6y_7\,W_3\, x_1,\] 
where the $W_i$ are cluster variables, and their corresponding arcs are illustrated in Figure \ref{Arcsingenus2}. The computation is shown in Figure \ref{snakegraphcalculusgenus2ccoeff}. The first equality in that figure is obtained by the grafting operation \cite[Section 2.9 case 2]{CS2}, the second equation by the self-grafting operation \cite[Section 3.4]{CS2}, and in the last equation the first two terms are obtained by grafting with a single edge \cite[Section 3.3 case 3]{CS2} and the third and fourth term by self-grafting.

This shows that $X_1\in \mathcal{A}_\bullet$ and  hence that $L\in\cala_\bullet$. 
\end{proof}

\begin{lem}
 $L\in\hat\cala$.
\end{lem}
\begin{proof}
The proof is exactly the same as in the genus 1 case.
\end{proof}
\begin{figure}[ht]
\begin{center}
\scalebox{0.57} { \Large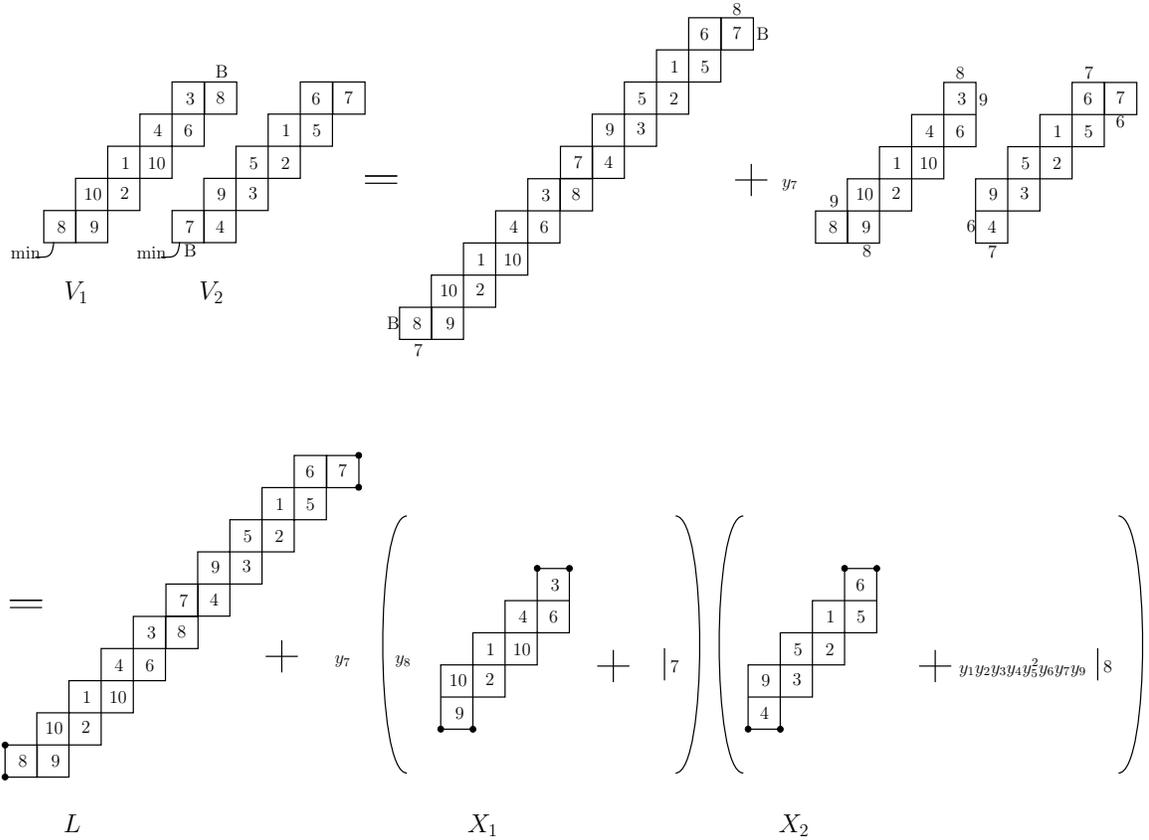}
 \caption{Snakegraph calculus used in the proof of Lemma \ref{lem 8} }
 \label{snakegraphcalculusgenus2acoeff}
 \end{center}
\end{figure}
\begin{figure}[ht]
\begin{center}
\scalebox{0.57} { \small 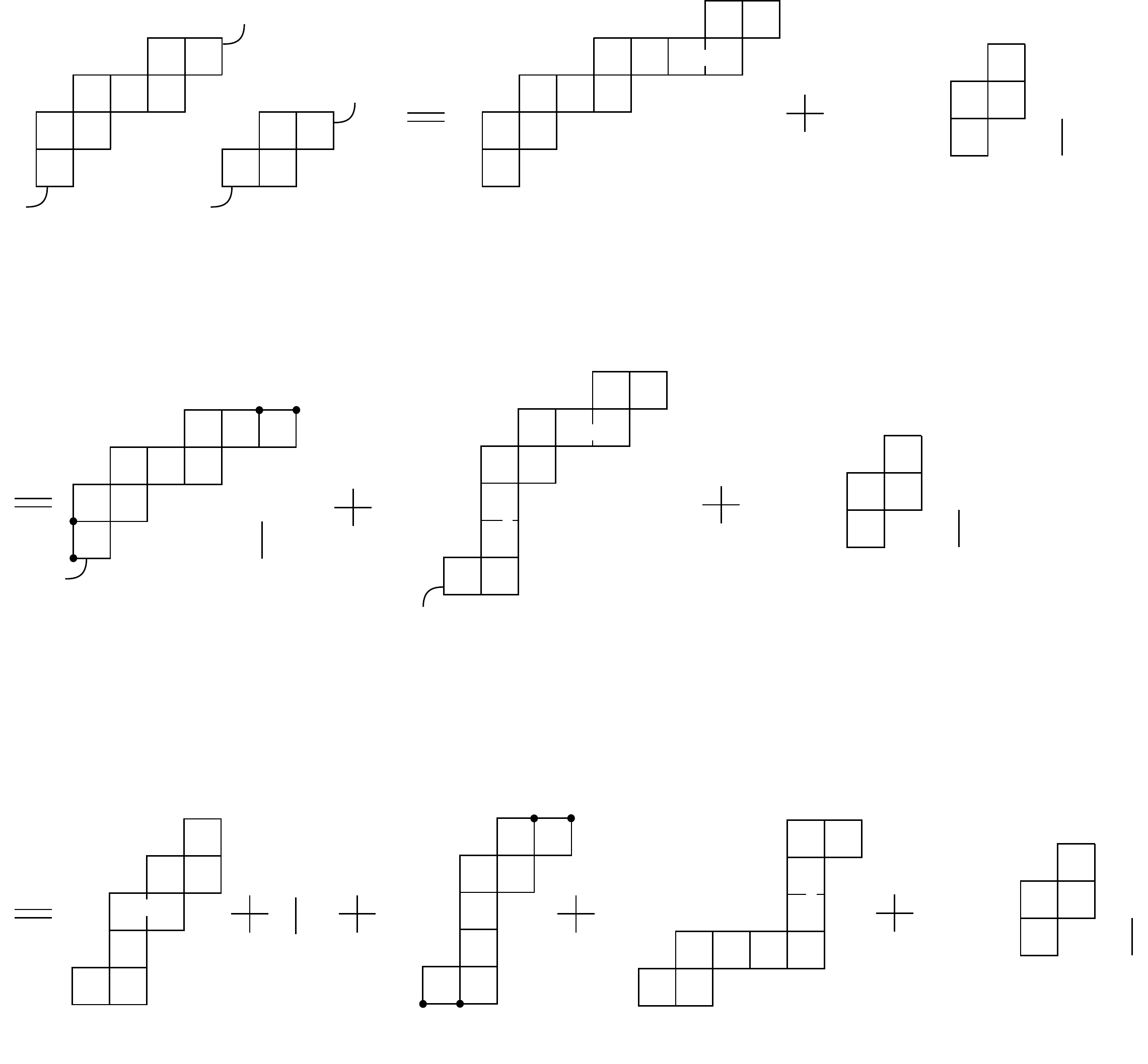}
 \caption{Snakegraph calculus showing that
 \[U_1U_2 = y_1\, W_1+x_7+ y_8\,X_1+ y_1y_8\,W_2+y_1y_5y_6y_7\,W_3\,x_1.\]}
 \label{snakegraphcalculusgenus2ccoeff}
 \end{center}
\end{figure}
\begin{figure}[ht]
\begin{center}
\scalebox{0.3} {\huge 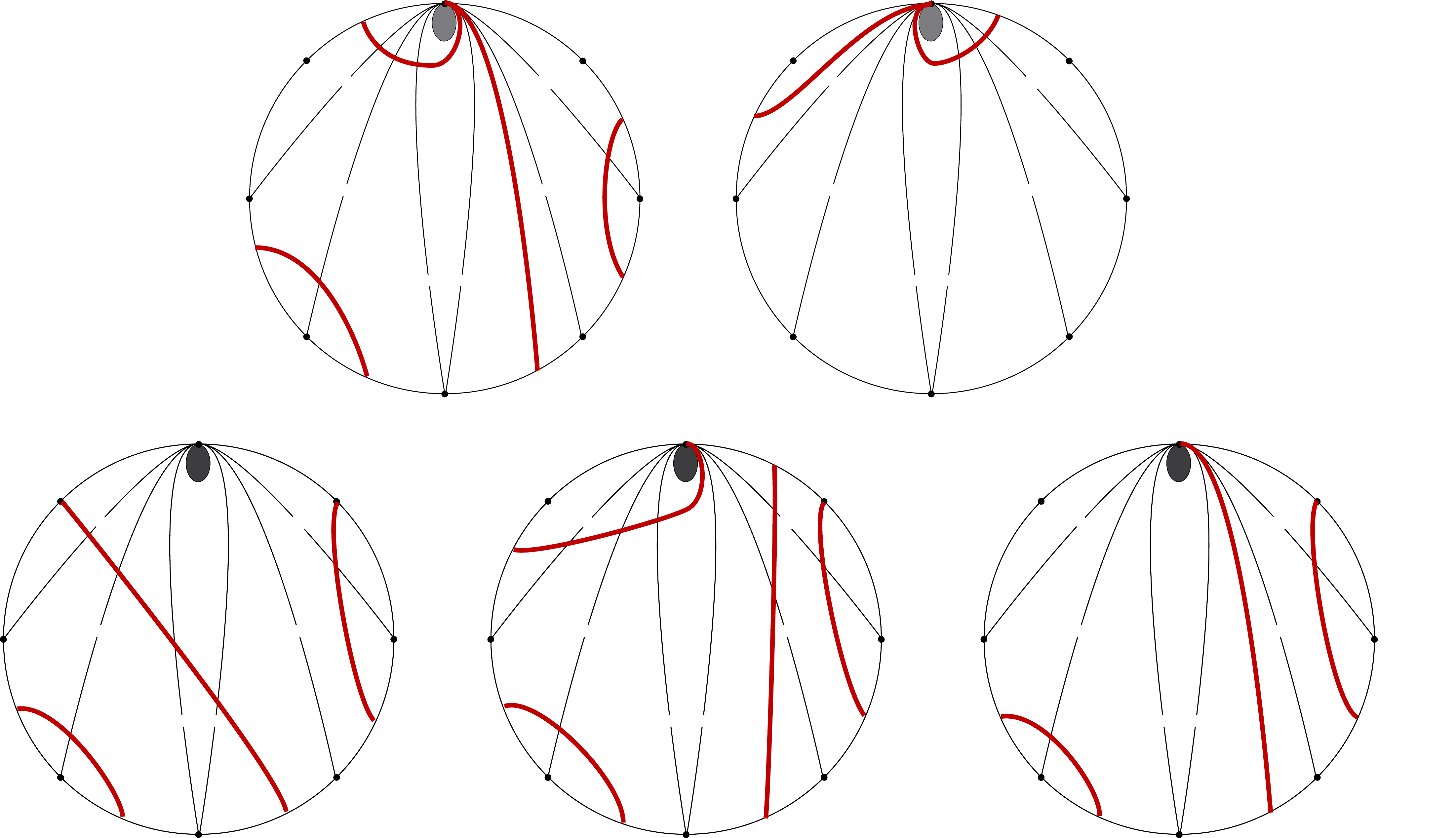}
 \caption{The arcs of the cluster variables involved in the computations in genus 2.}
 \label{Arcsingenus2}
 \end{center}
\end{figure}
%

\subsection{Higher genus}
In this section, let $(S,M)$ be a surface of genus $g>2$ with one boundary component and one marked point. Fix the triangulation {{$T$}} shown in Figure \ref{triangulationg>2}, and let $\mathbf{x}_T=(x_1,\ldots,x_{6g-2})$ be the corresponding cluster in $\cala_\bullet$. Note that the figure is for $g$ an even integer, but the case where $g$ is odd is similar; only the labeling changes slightly.

As in the genus 1 and 2 cases, let $L$ be the Laurent polynomial associated to the loop around the boundary. 

\begin{lem} 
$ L\in\cala_\bullet$.
\end{lem}

\begin{proof}
The proof is an adaptation of the genus 2 argument. 
The cluster variables $V_1$ and $V_2$ are now given by the zigzag snake graphs shown in Figure~\ref{fig V2}. The Laurent polynomials $X_1$, $X_2$ are {{again}} given by the band graphs obtained from the snake graphs of $V_1$ and $V_2$ by deleting the first and the last tile and then glueing. Again snake graph calculus shows that 
$$V_1\, V_2 = L + y_a (y_{a'}X_1  +{{x_a}})(X_2 + Y\,x_{a'}),$$
where $Y$ is a monomial in $\mathbf{y}$.  

Finally to show that $X_1$, and hence $X_2$, is in the cluster algebra, we use the cluster variables $U_1,U_2,W_1,W_2,W_3$ analogue to the genus 2 case to get
\[U_1U_2 = y_1\,W_1+ x_a+ y_{a'}\, X_1+ y_1y_{a'}\,W_2+ y_1y_py_o \cdots y_by_a\,W_3 \,x_1.\]
This shows that $X_1\in \mathcal{A}$ and we are done.
\end{proof}

\begin{lem}
 $L\in\hat\cala$.
\end{lem}
\begin{proof}
The proof is exactly the same as in the genus 1 case.
\end{proof}

\begin{figure}[ht]
\begin{center}
\scalebox{0.6} {\Large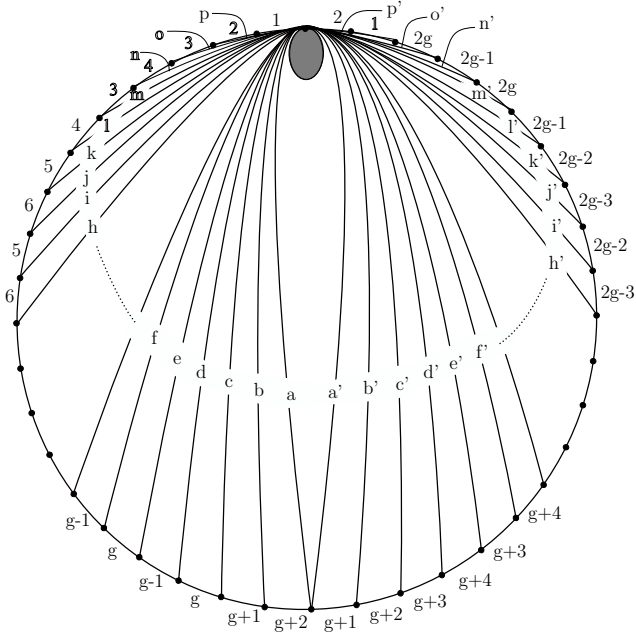}
 \caption{Triangulation of a surface of genus $g$ with one boundary component and one marked point}
 \label{triangulationg>2}
 \end{center}
\end{figure}

\begin{figure}[ht]
\begin{center}
\scalebox{.6} { \small 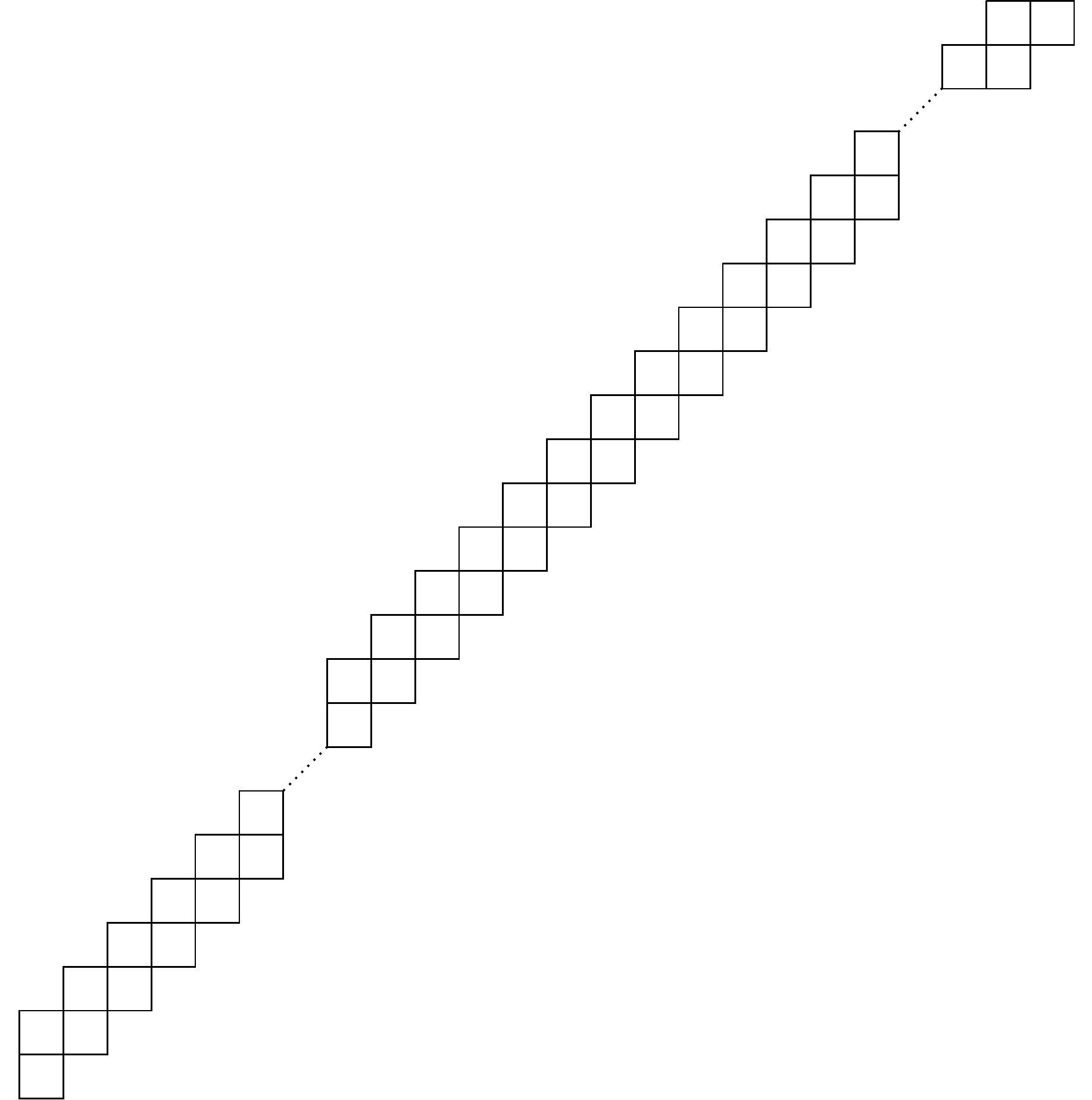 \hspace{-14cm}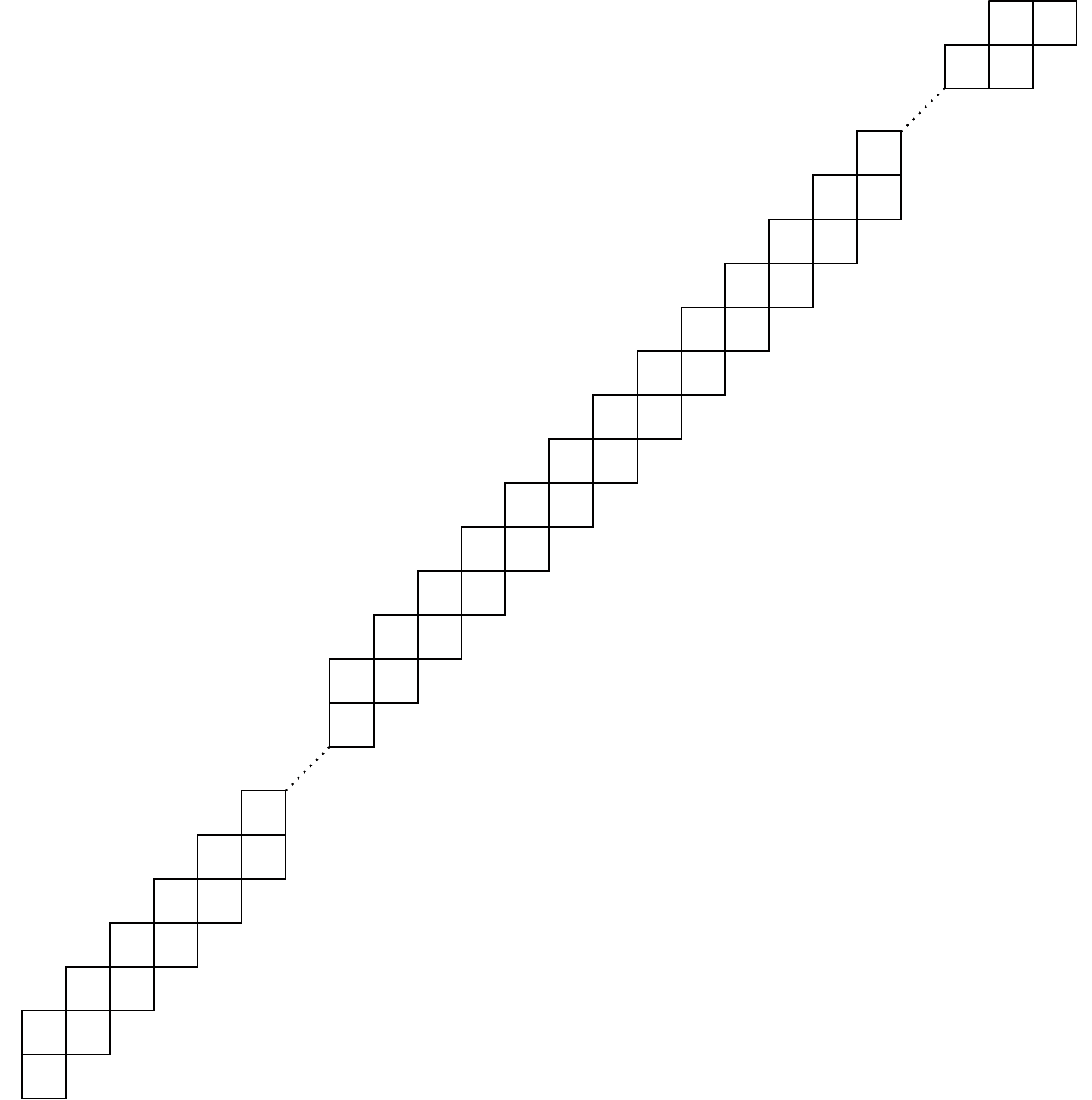}
%
%
%
%
 \caption{Snakegraphs of the cluster variables $V_1$ (left) and $V_2$ (right)}
 \label{fig V2}
 \end{center}
\end{figure}
%
%

\section{Proof of Lemma \ref{lem loop2}}\label{sect proof2}
Let $\zeta $ be an essential loop and let $x_\zeta $ be the Laurent polynomial associated to it. In view of Lemma \ref{lem loop} we may suppose that $\zeta $ is not the loop around the boundary. The proof uses a geometric argument with skein relations, and the relevant curves are illustrated in Figure \ref{fig skein}.
Choose a point $z$ on $\zeta$ and a simple curve $\za$ that goes from the marked point to the point $z$. Let $\zg$ be the arc obtained by the curve $\za\zeta\za^{-1}$. Let $\zd$ be the generalized arc that starts at the marked point, goes around the boundary twice  and then ends at the marked point.
The curves $\za,\zeta,\zd$ and $\zg$ are illustrated in the first row of Figure \ref{fig skein}. Let $x_\za,x_\zeta,x_\zd$ and $x_\zg$ be the corresponding Laurent polynomials.

We compute the product $x_\zd x_\zg$ using skein relations.  Thus each equation is obtained by smoothing a crossing of the curves. This computation is illustrated in the second and third row of Figure \ref{fig skein}.
The first term on the right hand side of the equation in the second row is a cluster variable $x_\ze$ given by the blue arc $\ze$ (the red arc is the boundary segment $B$)
and the second term on the right hand side still has a self-crossing. Applying the skein relations to that self-crossing produces the third row in the figure. In this row all red and green curves are boundary segments.
Thus we get 
\[x_\zd x_\zg = Y_1 \,x_\ze + Y_2\,x_\zeta  + Y_3\,x_\sigma \]
where $Y_i$ are some monomials in $\mathbf{y}$ and  $ x_\zg , x_\ze $ and $x_\sigma$ are cluster variables. Thus  in order to show that $x_\zeta\in \cala$, it suffices to show that $x_\zd\in \cala$. 
But it is shown in the last row of Figure \ref{fig skein} that $x_\zd = L+ 0$, where $L$ (in blue) is the essential loop around the boundary. Now the result follows from Lemma~\ref{lem loop}.
\qed  

\begin{figure}[ht]
\begin{center}
\scalebox{0.7} { \Large 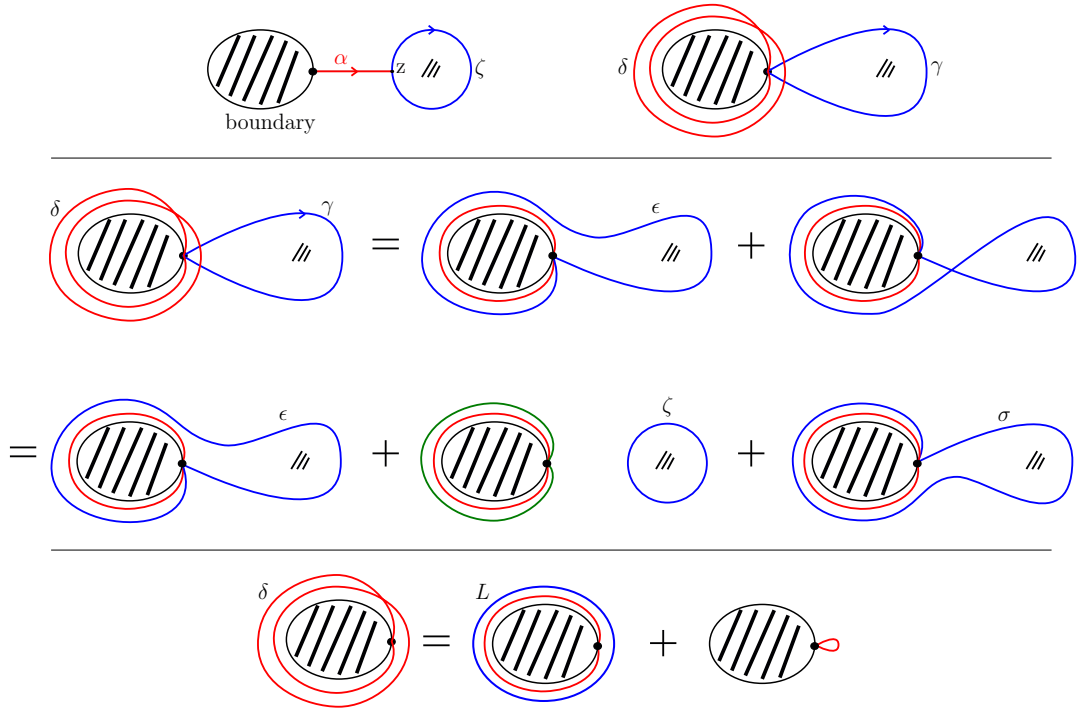}
 \caption{Proof of Lemma \ref{lem loop2}}
 \label{fig skein}
 \end{center}
\end{figure}

\section{Punctured surfaces with boundary}
In this section, we show how to reduce the case of punctured surfaces with exactly one marked point on the boundary to the case of surfaces with two marked points on the boundary.

Let $(S,M)$ be a punctured surface with boundary and one marked point on the boundary. Let $\cala$ be its cluster algebra and $\calu$ the corresponding upper cluster algebra.

\begin{prop}
 \label{prop 4} 
 $\cala =\calu$.
\end{prop}
\begin{proof}
 By \cite{MSp} it suffices to show that there is a triangulation whose quiver satisfies the Louise property. Let $p_0$ be the marked point on the boundary and let $p_1$ be a puncture. Let $\alpha,\beta$ be  two arcs going from $p_0$ to $p_1$ and $\gamma$ be the arc from $p_0$ to $p_0$ going around the puncture $p_1$ such that $\gamma$ cuts the surface into two pieces one of which  contains the boundary, the puncture $p_1$ and no other punctures, see {the picture on the left hand side of Figure \ref{fig puncture}.
If the genus of $S$ is zero and $p_1$ is the only puncture, then the cluster algebra is of type $\mathbb{A}_1$ and $\cala=\calu$. 

Otherwise,  $\gamma $ is not contractible, hence an arc,
and we can choose} a triangulation $T$ containing the arcs $\alpha,\beta,\gamma$. Then in the corresponding
 quiver $Q_T$ the vertices $\alpha$ and $\beta$ are sinks or sources. 
Now the quiver obtained by removing the vertices $\alpha$ and $\beta$ from the quiver $Q_T$ is the same as the quiver $Q_{T'}$ of the triangulation $T' =T-\{\alpha,\beta \}$ obtained by deleting the arcs $\alpha$ and $\beta$, of the surface $(S',M')$ obtained by replacing the puncture $p_1$ by a second point on the boundary, {see the picture on the right hand side of Figure \ref{fig puncture}.} By induction, $Q_{T'}$ is Louise since $(S',M')$  has two boundary points and one puncture less than $(S,M)$.
\begin{figure}[ht]
\begin{center}
\scalebox{0.7} { \Large 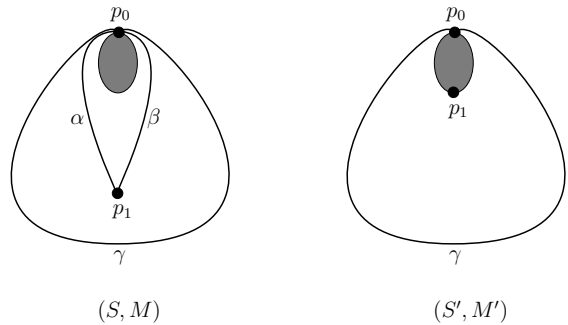}
 \caption{Proof of Proposition \ref{prop 4}}
 \label{fig puncture}
 \end{center}
\end{figure}
\end{proof}
{}

\end{document}